\documentclass[11pt, letterpaper, final]{article}
\usepackage{authblk}
\usepackage[letterpaper,margin=1in]{geometry}
\usepackage[utf8]{inputenc}
\usepackage[english]{babel}
\usepackage{amsmath,amssymb,color,amsthm}
\usepackage{mathrsfs}

\usepackage{graphicx}
\graphicspath{ {./Figures/} }
\usepackage{epsfig}

\usepackage{cite}
\usepackage{enumerate}
\usepackage{caption}
\usepackage{subcaption}
\usepackage{comment}
\usepackage{float}
\usepackage{changepage}

\numberwithin{equation}{section}
\usepackage{tikz}\usetikzlibrary{arrows}

\usepackage[colorlinks]{hyperref}
\usepackage[notref, notcite]{showkeys}

\renewcommand{\geq}{\geqslant}
\renewcommand{\leq}{\leqslant}

\newtheorem{theorem}{Theorem}[section]

\newtheorem{lemma}[theorem]{Lemma}

\newtheorem{corollary}[theorem]{Corollary}

\newtheorem*{main-theorem}{Main Theorem}
\newtheorem*{remark*}{Remark}
\numberwithin{equation}{section}

\DeclareMathOperator{\sgn}{sgn}
\DeclareMathOperator{\cn}{cn}

\DeclareMathOperator{\K}{K}

\usepackage{todonotes}

\title{Bounds on unstable spectrum \protect \\
for dispersive Hamiltonian PDEs}

\author[1]{Jared~C.~Bronski\thanks{Email: bronski@illinois.edu}}
\author[1]{Vera~Mikyoung~Hur\thanks{Email: verahur@illinois.edu}}
\author[1]{Sarah~E.~Simpson\thanks{Email: sarah97@illinois.edu}}
\affil[1]{Department of Mathematics, University of Illinois at Urbana-Champaign, \protect\\
Urbana, IL 61801, USA}

\begin{document}

\maketitle

\begin{abstract}
We study quasi-periodic eigenvalue problems that arise in the spectral stability analysis of periodic traveling waves to dispersive Hamiltonian PDEs. Specifically, we establish bounds on regions in the complex plane where eigenvalues may deviate from the imaginary axis, and estimates on the number of such unstable eigenvalues. These results hold when the dispersion relation grows faster than quadratically in the wave number and utilizes a Gershgorin disk theorem argument along with the symmetry of the spectrum. 
The results are applicable to a broad class of nonlinear dispersive equations, including the generalized Korteweg--de Vries, generalized Benjamin--Bona–-Mahony, and Kawahara equations.
\end{abstract}

\section{Introduction}\label{sec:intro}

We consider spectral problems of the form \begin{equation}\label{eqn:main} 
\lambda v=\mathcal{J}\mathcal{L}v,\qquad \lambda\in \mathbb{C},
\end{equation}
where \(\mathcal{J}\) is a linear skew-symmetric operator, \(\mathcal{L}\) is a linear self-adjoint operator representing the second variation of a suitable Hamiltonian, and \(\mathcal{J}\mathcal{L}\) has periodic coefficients. Such eigenvalue problems arise in the spectral stability analysis of periodic traveling waves to dispersive Hamiltonian partial differential equations (PDEs). Notable examples include the generalized Korteweg--de Vries (gKdV), generalized Benjamin--Ono (gBO), generalized Benjamin--Bona--Mahony (gBBM), and nonlinear Schrödinger equations. 
It is well-known that the spectrum of \eqref{eqn:main} appears in quartets---if \(\lambda\) is an eigenvalue, then so are \(-\lambda\), \(\overline{\lambda}\), and \(-\overline{\lambda}\)---indicating symmetry with respect to reflections across the real and imaginary axes. Our primary focus lies on quasi-periodic eigenvalue problems exhibiting this kind of symmetry, where any eigenvalues deviating from the imaginary axis signify instability.

For solitary waves or periodic traveling waves subject to same-period perturbations, the seminal work of Grillakis, Shatah, and Strauss \cite{grillakis1987stability,grillakis1990stability} established a relation between stability and the number of negative eigenvalues of the second variation of the reduced Hamiltonian, whose spectrum is discrete. This has developed into powerful analytical methodologies, leading to numerous applications. However, for periodic traveling waves under arbitrary perturbations, understanding the essential spectrum is exceedingly challenging. Nevertheless, for some integrable equations, it is possible to analytically confirm---through the Lax pair formalism---that the essential spectrum lies along the imaginary axis \cite{MR2552133,MR2724037,MR2812335, MR3624690}. Additionally, for various non-integrable equations, stability results of small amplitude waves have been established \cite{gallay2007stability,MR2357528, haragus2008spectra, MR2286731, MR2684614, leisman2021stability}. 
It is noteworthy that periodic traveling waves can become unstable under subharmonic perturbations. In particular, for a broad class of nonlinear dispersive equations, instability results have been established near the origin of the spectral plane for small values of the Floquet exponent, namely modulational instability \cite{MR2660515, Bronski.Johnson.Kapitula.2011, MR3183408, MR3194028, MR3298879, MR3408757, trichtchenko2018stability,stefanov2020small,gallay2007stability,haragus2008spectra,johnson2013stability}.

Significant efforts have been devoted to numerically compute essential spectrum for linearized operators about periodic traveling waves \cite{MR2273379, MR2323328, MR2200793,deconinck2011instability, sanford2014stability, MR4716934,barker2013nonlinear,kalisch2005numerical,angulo2014non,alharbi2020numerical,bona2007numerical}. However, these numerical approaches often rely on the implicit assumption that all unstable spectrum be confined within a bounded region in the complex plane, which warrants careful scrutiny. In fact, the first two authors together with their collaborator \cite{bronski2021superharmonic} demonstrated that this assumption does not hold for some regularized long-wave models, where the spectrum extends towards \(\pm i\infty\) along curves with nonzero real parts. Additionally, it is challenging to numerically distinguish whether spectrum with small but nonzero real part means genuine instability or it is merely an artifact of numerical error. For instance, the second author and her collaborator \cite{MR4600217} analytically confirmed the numerical prediction \cite{deconinck2011instability} of infinitesimally small 'bubbles' of instability for Stokes free-surface waves. 

Our objective is to establish explicit spectral bounds that can be directly utilized in numerical applications for a broad class of dispersive Hamiltonian PDEs. Specifically, we identify regions in the complex plane where eigenvalues potentially deviate from the imaginary axis and estimate the number of such unstable eigenvalues. 

For a fixed value of the Floquet exponent, ensuring a discrete spectrum, we rewrite \eqref{eqn:main} as
\[
\lambda v = \mathcal{JL}_0v + \mathcal{JL}_1v 
\]
or, equivalently,
\[
v = -(\mathcal{JL}_0 - \lambda)^{-1} \mathcal{JL}_1 v,
\]
where $\mathcal{JL}_0$ represents a constant-coefficient dispersion operator and $\mathcal{JL}_1$ has variable coefficients. If one can show that $\|(\mathcal{JL}_0-\lambda)^{-1}\mathcal{JL}_1\|<1$ in an appropriate function space, then $1+(\mathcal{JL}_0-\lambda)^{-1}\mathcal{JL}_1$ is invertible, implying that $\lambda \neq 0, \in \mathbb{C}$ must belong to the resolvent set. We shall demonstrate that this holds except for countably many Gershgorin (topological) disks, centered at purely imaginary eigenvalues of $\mathcal{JL}_0$, with radii depending on $\mathcal{JL}_1$ appropriately. 

Additionally, if there exists an eigenvalue $\lambda$ of \eqref{eqn:main} such that
\[
|\lambda - i\alpha| < \epsilon \quad \text{for some $\alpha \in \mathbb{R}$ and some $\epsilon > 0$}
\]
and if
\[
|\lambda - \lambda'| > 2\epsilon \quad \text{for all eigenvalues $\lambda' \neq \lambda$},
\]
then $\lambda \in i\mathbb{R}$ must hold by Hamiltonian symmetry. In other words, if an eigenvalue is close to the imaginary axis but isolated, then it must be purely imaginary. Consequently, if one can show that each Gershgorin disk contains precisely one eigenvalue and is disjoint from other disks, then the eigenvalue within the disk must be purely imaginary. For several examples, we shall demonstrate this for all eigenvalues for sufficiently large wave numbers, thereby ensuring that only finitely many eigenvalues can lie off the imaginary axis. 

For instance, for the gKdV equation (see \eqref{eqn:gKdV}), the $k$-th Gershgorin disk is centered at $i(k^3-ck)$, representing the dispersion relation, with a radius of $O(k)$, where $k \in \mathbb{R}$ is the wave number. The distance from the $k$-th disk to its adjacent disks is $O(k^2)$, ensuring that the disks remain disjoint from the others for sufficiently large $|k|$. Consequently, the eigenvalues within the disks must be purely imaginary. 

The gBO equation (see \eqref{eqn:BO}) represents a borderline case, where the distance from the $k$-th Gershgorin disk to the adjacent disks is $O(k)$, and the radius of the disk is also $O(k)$. Notably, we shall demonstrate that the Gershgorin disks for the BO equation (quadratic power-law nonlinearity) overlap except potentially at one point. However, numerical experiments suggest that the spectrum is confined to regions significantly smaller than those predicted by the Gershgorin disk theorem argument. 

The idea of using topological disks to establish spectral bounds traces back to Howard's semi-circle theorem \cite{howard1961note} (see also \cite{banerjee1988reducing,chimonas1970extension,waleffe2019semicircle}), which identifies a region where the complex wave velocity for an unstable mode may reside.
The Gershgorin disk theorem \cite{gerschgorin1931sa,varga2010gervsgorin} has been instrumental for stability analysis in various situations. See, for instance, \cite{FL1989,FL1991}, where the authors established general results on the number and possible locations of eigenvalues of operators in Banach spaces. The key innovation here, which allows us to count the possible dimensions of off-axis eigenvalues, is the incorporation of symmetry. Specifically, if there is one eigenvalue in an isolated disk, then it must necessarily lie on the axis of symmetry. 
Perhaps the work most closely related to ours is the recent study \cite{gaebler2021nls}, where the authors employed semigroup techniques to show that any eigenvalues off the imaginary axis must lie within a strip of fixed width about the real axis. However, this does not provide a means to count the number of such eigenvalues. We discuss the similarities and differences between the two approaches in greater detail in Section \ref{sec:conclusion}.

Our results can help reduce the risk of overlooking unstable eigenvalues in numerical experiments and decrease computational costs by limiting required numerical computations to locate such spectra. More importantly, they represent an important preliminary step towards establishing spectral stability proofs through rigorous numerics for non-integrable PDEs \cite{BBHH}. 


\section{The main result}\label{sec:Gershgorin}



We consider the eigenvalue problem 
\begin{equation}\label{eqn: General eigenvalue problem}
\lambda v=(i\omega(-i\partial_x)+c\partial_x+\partial_x Q(x))v,
\quad \lambda \in \mathbb{C},
\end{equation}
subject to the boundary condition
\begin{equation}\label{eqn: Floquet quasiperiodic BCs}
v(T)=e^{2 \pi i \mu } v(0),\quad \mu\in(-\tfrac{1}{2},\tfrac{1}{2}],
\end{equation}
along with its derivatives. Here $\omega(k)$, $k\in\mathbb{R}$, is a real-valued function, representing the dispersion relation, and it satisfies $\omega(-k)=-\omega(k)$; $c \in \mathbb{R}$ is the wave speed; and $Q(x)$, $x\in\mathbb{R}$, is a real-valued function, satisfying $Q(x+T)=Q(x)$ for some $T>0$, the period. We assume that $Q(x)$ has zero mean over one period, with the mean absorbed into $c$. We introduce the notation
\[
\|\widehat{Q}\|_1=\sum_{k\neq0,\in\mathbb{Z}} |\widehat{Q}_k|, \quad \text{where}\quad Q(x)=\sum_{k\neq 0,\in\mathbb{Z}} \widehat{Q}_ke^{\frac{2\pi ikx}{T}}.
\]
The quasi-periodic eigenvalue problem \eqref{eqn: General eigenvalue problem}-\eqref{eqn:  Floquet quasiperiodic BCs} arises in the stability analysis of periodic traveling waves to nonlinear dispersive equations. Notable examples include:
\begin{itemize}
\item the generalized KdV equation: $\omega(k)=k^3$,
\item the Benjamin--Ono equation: $\omega(k) =  k^2 \sgn(k)$,
\item the Kawahara equation: $\omega(k) = k^5 + \alpha  k^3$, $\alpha\in\mathbb{R}$.
\end{itemize}

We state the main result of the section. 

\begin{theorem}\label{Thm: General Gershgorin discs}
For each $\mu\in(-\frac{1}{2},\frac{1}{2}]$, a nonzero eigenvalue of \eqref{eqn: General eigenvalue problem}-\eqref{eqn: Floquet quasiperiodic BCs} must be contained in ${\displaystyle \bigcup_{k\in\mathbb{Z}} D_k(\mu)}$, where the $k$-th Gershgorin (topological) disk is defined as 
\begin{equation}\label{eqn: General Gershgorin disc formula}
D_k(\mu)=\Big\{\lambda \in \mathbb{C}:\Big|\lambda-i\omega\Big(\frac{2\pi}{T}(k+\mu)\Big)- \frac{2\pi ic}{T}(k+\mu)\Big| \leq \frac{2\pi}{T}|k+\mu| \|\widehat{Q}\|_1 \Big\}.
\end{equation}
Furthermore, any connected component consisting of $n$ intersecting Gershgorin disks must contain precisely $n$ eigenvalues. In particular, if a Gershgorin disk is disjoint from all other disks, then it contains exactly one eigenvalue.
\end{theorem}

\begin{proof}
For a fixed value of $\mu\in(-\frac{1}{2},\frac{1}{2}]$, ensuring that the spectrum of \eqref{eqn: General eigenvalue problem}-\eqref{eqn: Floquet quasiperiodic BCs} is discrete, we rewrite \eqref{eqn: General eigenvalue problem} as
\[
v=(\lambda-i \omega(-i \partial_x)-c\partial_x)^{-1}\partial_x(Q(x)v).
\]
If one can demonstrate that $\|(\lambda-i\omega(-i\partial_x)-c\partial_x)^{-1}\partial_xQ(x)\|<1$ in an appropriate function space, then $1-(\lambda-i\omega(-i\partial_x)-c\partial_x)^{-1}\partial_xQ(x)$ is invertible, implying that $\lambda\neq 0,\in\mathbb{C}$ must belong to the resolvent set of $i\omega(-i\partial_x)+c\partial_x+\partial_x Q(x)$. 

We expand $v$ in a Fourier series, leading to
\[
\sum_{k\in\mathbb{Z}} \widehat{v}_k e^{\frac{2 \pi i (k+\mu) x}{T}}=(\lambda-i\omega(-i\partial_x)-c\partial_x)^{-1}\partial_x\Big(\sum_{k\in\mathbb{Z}} \sum_{\ell\in\mathbb{Z}} \widehat{Q}_k \widehat{v}_\ell e^{\frac{2 \pi i (k+\ell+\mu)x}{T}} \Big).
\]
A straightforward calculation reveals
\begin{equation}\label{eqn:v2}
\hat{v}_k=\frac{\frac{2\pi i}{T}(k+\mu)}{\lambda-i\omega\big( \frac{2\pi}{T} (k+\mu)\big)- \frac{2\pi ic}{T}(k+\mu)} \sum_{\ell\in\mathbb{Z}} \widehat{Q}_{k-\ell}\widehat{v}_\ell,
\end{equation}
whence
\[
\max_{k\in\mathbb{Z}}|\widehat{v}_k| \leq \max_{k\in\mathbb{Z}}\left|\frac{\frac{2\pi i}{T}(k+\mu)}{\lambda-i\omega\big( \frac{2\pi}{T} (k+\mu)\big)- \frac{2\pi ic}{T}(k+\mu)} \right|\max_{\ell\in\mathbb{Z}} |\widehat{v}_\ell|\sum_{\ell\in\mathbb{Z}} |\widehat{Q}_{k-\ell}|.
\]
Consequently, if \eqref{eqn:v2} has a nontrivial solution then, necessarily,
\begin{equation}
1\leq \max_{k\in\mathbb{Z}}\left|\frac{\frac{2\pi i}{T}(k+\mu)}{\lambda-i\omega\big( \frac{2\pi}{T} (k+\mu)\big)- \frac{2\pi ic}{T}(k+\mu)} \right|\sum_{\ell\in\mathbb{Z}} |\widehat{Q}_{k-\ell}|,
\label{eqn:MainG}
\end{equation}
which gives the region exterior to the Gershgorin (topological) disks, centered at $i\omega(\frac{2\pi}{T} (k+\mu))+\frac{2\pi ic}{T} (k+\mu)$---along the imaginary axis---with radii $\frac{2\pi}{T}|k+\mu| \|\widehat Q\|_1$, where $k\in\mathbb{Z}$. That implies that any $\lambda\neq 0,\in\mathbb{C}$ outside the union of these disks must belong to the resolvent set of $i\omega(-i\partial_x)+c\partial_x+\partial_x Q(x)$. Therefore, the eigenvalues of \eqref{eqn: General eigenvalue problem}-\eqref{eqn: Floquet quasiperiodic BCs} must lie in $\bigcup_{k\in\mathbb{Z}} D_k(\mu)$, where $D_k(\mu)$ is in \eqref{eqn: General Gershgorin disc formula}.

To see that a connected component of $n$ intersecting Gershgorin disks contains precisely $n$ eigenvalues, we consider 
\begin{equation}\label{eqn:tau}
\lambda v =i\omega(-i\partial_x)v + cv_x + \tau(Q(x)v)_x, \quad  \tau\in[0,1].
\end{equation}
It follows from perturbation theory (see, for instance, \cite{Kato}) that under some assumptions\footnote{It suffices to assume that the potential term $\partial_x Q(x)$ be a relatively compact perturbation of the dispersion operator $i \omega(-i\partial_x)$. This necessitates some growth condition on $\omega(k)$ for $|k|\gg1$, which is satisfied in all examples considered herein.}, the eigenvalues of \eqref{eqn:tau} vary continuously with respect to $\tau$. When $\tau=0$, the eigenvalues of \eqref{eqn:tau} are $i\omega(\frac{2\pi}{T} (k+\mu))+\frac{2\pi ic}{T} (k+\mu)$, where $k\in\mathbb{Z}$, with precisely one eigenvalue per disk. Therefore, when $\tau=1$, there must be exactly $n$ eigenvalues in the union of the $n$ Gershgorin disks. This completes the proof. We remark that there are $n$ eigenvalues in the connected component of $n$ intersecting disks, but not necessarily one eigenvalue per disk unless the disks are disjoint. 
\end{proof}

We turn to establishing a sufficient condition ensuring that the Gershgorin disks $D_k(\mu)$ become asymptotically disjoint as $|k|\to \infty$, thereby offering an upper bound on the number of overlapping disks and, consequently, an upper bound on the number of eigenvalues which may deviate from the imaginary axis. 

\begin{theorem}\label{Thm: General disjoint disc condition}
For each $\mu\in(-\frac{1}{2},\frac{1}{2}]$, if
\begin{equation}\label{eqn: General disjoint disc condition}
\begin{aligned}
\Big|\omega\Big(\frac{2\pi}{T} (k+1+\mu)\Big)
+\frac{2\pi c}{T}(k+1+\mu)
&-\omega\Big(\frac{2\pi}{T} (k+\mu)\Big)-\frac{2\pi c}{T}(k+\mu)\Big| \\
&\qquad>\frac{2\pi}{T}(|k+1+\mu|+|k+\mu|)\|\widehat Q\|_1 
\end{aligned}
\end{equation}
for all $k\in\mathbb{Z}$ such that $k>k^\ast$ or $k<k_\ast$ for some $k^\ast$, $k_\ast \in \mathbb{Z}$, then the eigenvalues of \eqref{eqn: General eigenvalue problem}-\eqref{eqn: Floquet quasiperiodic BCs} lying outside of ${\displaystyle \Big(\bigcup_{k=0}^{k^\ast} D_k(\mu)\Big)\bigcup \Big(\bigcup_{k=k_\ast}^{0} D_k(\mu)\Big)}$ must be purely imaginary. 
\end{theorem}

\begin{proof}
Recall from Theorem~\ref{Thm: General Gershgorin discs} that any nonzero eigenvalue of \eqref{eqn: General eigenvalue problem}-\eqref{eqn: Floquet quasiperiodic BCs} must be contained in the union of the Gershgorin disks $\bigcup_{k\in\mathbb{Z}} D_k(\mu)$. We observe that \eqref{eqn: General disjoint disc condition} indicates that the distance between the centers of the $k$-th and $(k+1)$-st disks is greater than the sum of their radii, ensuring that the disks are disjoint. Additionally, recall from Theorem~\ref{Thm: General Gershgorin discs} that each Gershgorin disk that is disjoint from all other disks contains exactly one eigenvalue, which must be purely imaginary by Hamiltonian symmetry. Indeed, if $\lambda$ is an eigenvalue, then so is $-\overline{\lambda}$, and the disks are centered on the imaginary axis. This completes the proof.
\end{proof}

\begin{remark*}\rm
The operator $i \omega(-i\partial_x) + c \partial_x$ represents the dispersion relation, so that $\omega'(k) + c$ corresponds to the group velocity, while the operator $\partial_xQ(x)$ accounts for nonlinear effects. Theorem \ref{Thm: General disjoint disc condition} thus ensures stability when group velocity dispersion dominates nonlinearity. 

In particular, if $\omega'(k)+c>\omega_0>0$ for some $\omega_0$---that is, the dispersion relation is monotonically increasing with respect to the wave number---and if $\omega'(k)$ increases faster than linearly for sufficiently large $|k|$, then {\em all Gershgorin disks are disjoint} for sufficiently small $\|\widehat Q\|_1$. Conversely, if the dispersion relation is not monotone with the wave number, then there will typically be at least a finite number of overlapping Gershgorin disks. This is related to the idea of Krein signature, where the eigenvalues associated with portions of the dispersion curve that `goes the wrong way' can be interpreted as having negative Krein signature, potentially leading to instability. 

In what follows, we may restrict our attention to the case where $k>0$ as the case of $k<0$ can be treated analogously. In many examples of interest, the dispersion relation is odd with respect to the wave number, and the results exhibit antisymmetry. For instance $k_* = -k^*$. 

Additionally, it is convenient to work with $\kappa = \frac{2\pi k}{T}$ rather than $k$.
\end{remark*}

For numerical applications, achieving the tightest possible bound may not always be necessary; instead, having a readily available bound that is easy to work with can be more practical. We give a simpler sufficient condition ensuring that the Gershgorin disks become asymptotically disjoint.

\begin{corollary}\label{cor: General disjoint disc condition}
Suppose that $\omega'(k),\omega''(k),\omega'''(k)\geq 0$ for all $k\in\mathbb{R}$. If
\begin{equation}
\omega'(\kappa) + c \geq 2\kappa \|\widehat Q\|_1
\end{equation}
for all $\kappa \geq \kappa^*$, where $\kappa=\frac{2\pi k}{T}$ and $\kappa^*=\frac{2\pi k^*}{T}$ for some $k^*>0,\in\mathbb{Z}$, then $D_k(\mu)$ are disjoint from all other Gershgorin disks provided $|k| \geq k^*+1$ for all $\mu \in (-\tfrac12,\tfrac12]$.

In particular, if $\omega(k)=O(|k|^d)$ as $|k|\to\infty$ for some $d>2$, then $D_k(\mu)$ are disjoint for $|k|>k^*$ for $\mu \in (-\tfrac12,\tfrac12]$, where 
\begin{equation}\label{eqn:kBound}
k^*=\begin{cases}
{\displaystyle \Big(\frac{T}{2\pi}\Big)^{\frac{d-1}{d-2}}
\Big(\frac{2 \|\widehat Q\|_1-c}{d}\Big)^{\frac{1}{d-2}} }, &c<0\\
{\displaystyle \Big(\frac{T}{2\pi}\Big)^{\frac{d-1}{d-2}}\Big(\frac{2 \|\widehat Q\|_1}{d}\Big)^{\frac{1}{d-2}}}, &c>0
\end{cases}
\end{equation}
whence 
\[
(\text{$\#$ of unstable eigenvalues})\leq 2 k^*
\]
for each $\mu\in(-\frac12,\frac12]$.
\end{corollary}

\begin{proof}
Recall Jensen's inequality: if $f''(y)>0$ over the interval $(a,b)$ then  
\[
\int_a^b f(y)~dy>f\Big(\frac{a+b}{2}\Big).
\]
Consequently,
\begin{multline*}
\omega\Big(\frac{2 \pi}{T}(k+1+\mu)\Big) + \frac{2\pi c}{T}(k+1+\mu) 
- \omega\Big(\frac{2 \pi}{T}(k+\mu)\Big) - \frac{2 \pi c}{T}(k+\mu) \\
= \int_{\frac{2 \pi}{T}(k+\mu)}^{\frac{2 \pi}{T}(k+\mu+1)} (\omega'(\kappa) +c)~d\kappa 
\geq \omega'\Big(\frac{2 \pi }{T}(k+\tfrac{1}{2}+\mu)\Big) + c.
\end{multline*}
Recall Theorem \ref{Thm: General disjoint disc condition}, and it suffices to show that 
\[
\omega'\Big(\frac{2 \pi}{T}(k+\tfrac12+\mu)\Big) + c \geq \frac{4 \pi}{T}(k+\tfrac12+\mu)\|\widehat Q\|_1.
\]
Let $\kappa=\frac{2 \pi }{T}(k+\mu+\frac12)$ and observe that $\mu+\frac12>0$. We can bound $-c + 2\kappa \| \widehat Q\|_1$ for simplicity by twice the larger term, and \eqref{eqn:kBound} follows. This completes the proof. 
\end{proof}

\section{Examples}

We apply Theorems~\ref{Thm: General Gershgorin discs} and \ref{Thm: General disjoint disc condition} to gKdV, Kawahara, and BBM equations.

\subsection{The gKdV equation}

We begin by considering the spectral problem for the gKdV equation
\begin{equation}\label{eqn:gKdV}
\lambda v=(-\partial_{xxx}+c\partial_x+\partial_xQ(x))v, \quad \lambda\in\mathbb{C},
\end{equation}
where $c\in\mathbb{R}$ is the wave speed, and $Q(x)$ is a real-valued function satisfying $Q(x+T)=Q(x)$ for some $T>0$, the period. Without loss of generality, we assume $\int_0^T Q(x)~dx=0$, absorbing the mean value of $Q(x)$ into $c$. 

For each $\mu\in(-\frac12,\frac12]$, the $k$-th Gershgorin disk for \eqref{eqn:gKdV} and \eqref{eqn: Floquet quasiperiodic BCs} is defined as 
\[
D_k(\mu)=\Big\{\lambda\in\mathbb{C} : \Big| \lambda-i\Big(\frac{2\pi}{T}\Big)^3(k+\mu)^3-\frac{2\pi ic}{T}(k+\mu)\Big| 
\leq \frac{2\pi }{T}|k+\mu|\|\widehat Q\|_1 \Big\},\quad k\in\mathbb{Z}.
\]
See \eqref{eqn: General Gershgorin disc formula}, where the dispersion relation is $\omega(k)=k^3+ck$. It follows from Theorem \ref{Thm: General Gershgorin discs} that for each $\mu\in(-\frac12,\frac12]$, the eigenvalues of \eqref{eqn:gKdV} and \eqref{eqn: Floquet quasiperiodic BCs} lie in the union of the Gershgorin disks $\bigcup_{k\in\mathbb{Z}} D_k(\mu)$. Furthermore, any connected component of $n$ intersecting Gershgorin disks contains precisely $n$ eigenvalues, and if a Gershgorin disk is disjoint from all others, then it contains exactly one eigenvalue. In particular, for sufficiently large $|k|$, the Gershgorin disks become disjoint. Indeed, $D_k(\mu)$ is centered at $i(\frac{2\pi }{T})^3(k+\mu)^3+\frac{2\pi ic}{T}(k+\mu)=i\frac{8\pi^3}{T^3}k^3+O(k^2)$ as $|k|\to\infty$, whence the distance between the centers of the adjacent disks is $\frac{24\pi^3}{T^3}k^2+O(k)$ as $|k|\to \infty$, while the radius of $D_k(\mu)$ is $\frac{2\pi }{T}|k+\mu|\|\widehat Q\|_1$.

\begin{lemma}\label{Lemma:KdVdisjointcondition}
For each $\mu\in(-\frac12,\frac12]$, the $k$-th Gershgorin disk for \eqref{eqn:gKdV} and \eqref{eqn: Floquet quasiperiodic BCs} is disjoint from all others, provided $|k|>k^*$, where  
\begin{equation}\label{eqn:k*(gKdV)}
k^*=
\begin{cases}
{\displaystyle \Big(\frac{T}{2\pi}\Big)^2\frac{2 \|\widehat Q\|_1-c}{d}}, &c<0\\
{\displaystyle \Big(\frac{T}{2\pi}\Big)^2\frac{2 \|\widehat Q\|_1}{d}}, &c>0,
\end{cases}
\end{equation}
implying that the number of eigenvalues which may deviate from the imaginary axis $\leq 2k^*$.

In particular, {\em all} Gershgorin disks become disjoint from others, so that all eigenvalues are confined to the imaginary axis, provided
\begin{equation}\label{eqn:Q(gKdV)}
\|\widehat Q\|_1^2 < 3 c \Big(\frac{2\pi}{T}\Big)^2.
\end{equation}
\end{lemma}

Here we prioritize simplicity over sharpest possible result.

\begin{proof}
For \eqref{eqn:k*(gKdV)}, see Corollary~\ref{Thm: General disjoint disc condition}. For \eqref{eqn:Q(gKdV)}, we use $(y+1)^3 - y^3 > 3 (y+\frac12)^2$ and note that 
\[
3\Big(\frac{2\pi}{T}\Big)^3 (y+\tfrac12)^2 + \frac{4\pi}{T}(y+\tfrac12) + \frac{2\pi}{T} + \frac{2\pi c}{T}
\] 
has discriminant 
\[
4\Big(\frac{2\pi}{T}\Big)^2 \Big(\|\widehat Q\|_1^2 - 3 c \Big(\frac{2\pi}{T}\Big)^2\Big).
\]
If the discriminant is negative, then the quadratic polynomial remains positive, implying that all Gershgorin disks are mutually disjoint. This completes the proof.
\end{proof}

To illustrate our findings numerically, we take the mKdV equation
\begin{equation}\label{eqn:mKdV}
u_t=u_{xxx}+(u^3)_x, 
\end{equation}
which admits a family of periodic traveling waves of the form 
\begin{equation}\label{eqn:phi(gKdV)}
u(x,t)=\phi(x+ct)=\sqrt{2 m}A\cn(Ax + (2m-1)A^2 t)
\end{equation}
for some $A>0$, where $\cn(x,m)$ denotes the Jacobi elliptic function, and $0<m<1$ the elliptic modulus. The period is $4\K(m)$, where $\K(m)$ is the complete elliptic integral of the first kind. Linearizing \eqref{eqn:mKdV} about \eqref{eqn:phi(gKdV)}, in the frame of reference moving at the speed $-c$, we arrive at the spectral problem
\begin{equation}\label{eqn:mKdVStable}
-\lambda v = v_{xxx} - c v_x + (3\phi(x)^2v)_x, \quad \lambda\in\mathbb{C}.
\end{equation}
The period of the potential is $2\K(m)$, but we solve \eqref{eqn:mKdVStable} over the interval $[0,4\K(m)]$, the period of the traveling wave solution. 

We present our numerical results for $A=1$ and $m=\frac12$, for which the periodic traveling wave to \eqref{eqn:phi(gKdV)} oscillates between $\pm 1$ with a period approximately $7.416$. Solving \eqref{eqn:mKdVStable} using a spectral method, with known formulae for the Fourier coefficients (see, for instance, \cite{kiper1984fourier}), we numerically compute the spectrum. We numerically find that $\|\widehat{3\phi^2}\|_1\approx 1.63$ (excluding the mean value) and $c \approx 1.371$. Applying the result of Lemma \ref{Lemma:KdVdisjointcondition} (see also Corollary \ref{cor: General disjoint disc condition}), we numerically find that the $k$-th and $(k+1)$-st Gershgorin disks are disjoint provided 
\[
|k| > k^* \approx 2.15,
\] 
whereby for each $\mu\in(-\frac12,\frac12]$, at most seven Gershgorin disks can overlap. Specifically, the third disk must be disjoint from the fourth, although it might still overlap with the second. Applying the result of Theorem \ref{Thm: General disjoint disc condition}, we numerically find that the $k$-th and $(k+1)$-st Gershgorin disks are disjoint provided  
\[
|k+\mu| \gtrsim 1.32.
\]
Since $\mu \in (-\frac12,\frac12]$, this holds for
\[
|k| \gtrsim 1.82,
\]
implying that a Gershgorin disk is disjoint from others provided $|k|>2$. Consequently, there are at most five overlapping disks. Therefore, for each $\mu \in (-\frac12,\frac12]$, there can be at most four eigenvalues off the imaginary axis.


\begin{figure}\centering
    \begin{subfigure}[b]{.4\columnwidth}
\includegraphics[width=\textwidth]{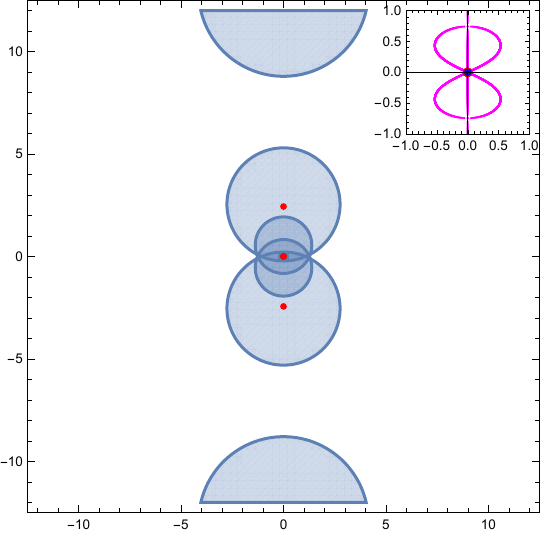}
        \caption{$\mu=0$}
        \label{fig:KdV0}
    \end{subfigure}
    \begin{subfigure}[b]{.4\columnwidth}
        \includegraphics[width=\textwidth]{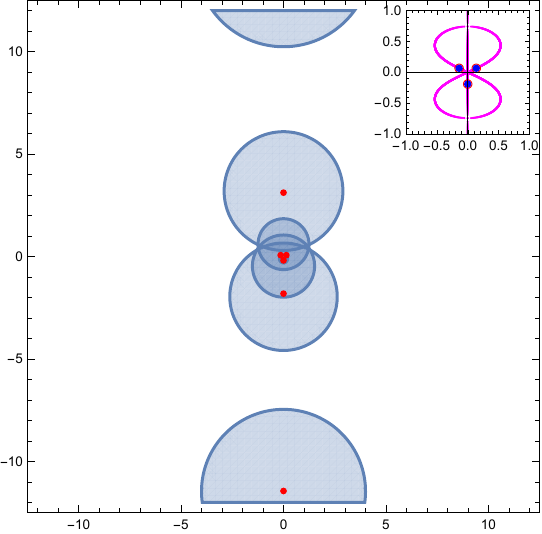}
        \caption{$\mu=0.1$}
        \label{fig:KdV1}
    \end{subfigure}
     \begin{subfigure}[b]{.4\textwidth}
        \includegraphics[width=\textwidth]{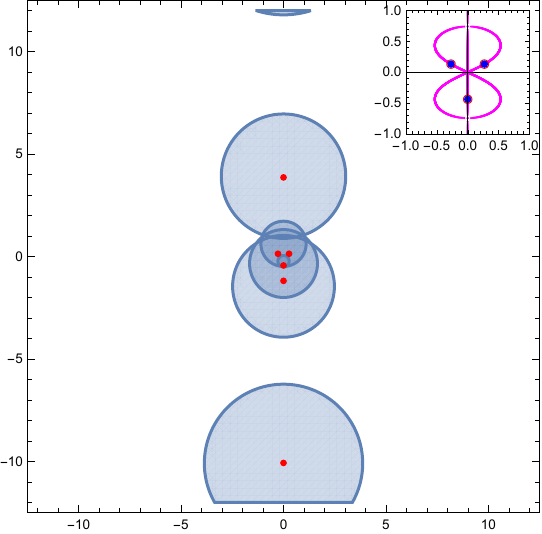}
        \caption{$\mu=0.2$}
        \label{fig:KdV2}
    \end{subfigure}
    \begin{subfigure}[b]{.4\columnwidth}
        \includegraphics[width=\textwidth]{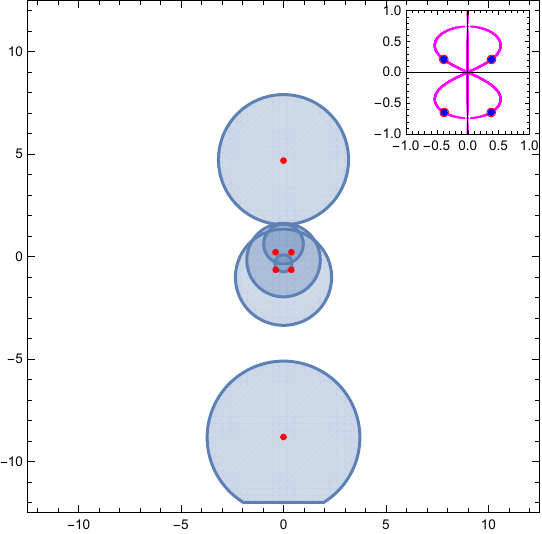}
        \caption{$\mu=0.3$}
        \label{fig:KdV3}
    \end{subfigure}  
    \begin{subfigure}[b]{.4\columnwidth}
        \includegraphics[width=\textwidth]{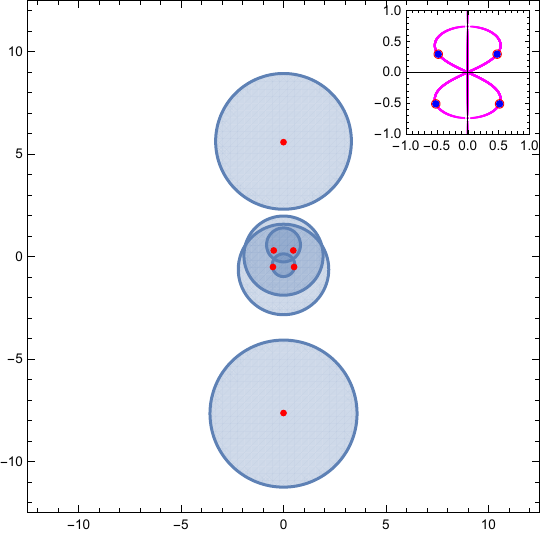}
        \caption{$\mu=0.4$}
    \label{fig:KdV4}
    \end{subfigure}   
        \begin{subfigure}[b]{.4\columnwidth}
        \includegraphics[width=\textwidth]{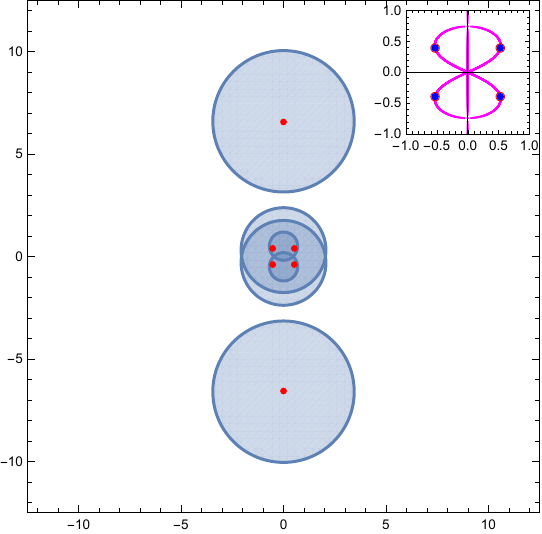}
        \caption{$\mu=0.5$}
    \label{fig:KdV5}
    \end{subfigure}
    \caption{The Gershgorin disks and the eigenvalues for \eqref{eqn:mKdVStable} and \eqref{eqn: Floquet quasiperiodic BCs} for several values of $\mu$, where $\phi$ is in \eqref{eqn:phi(gKdV)} for $A=1$ and $m=\frac12$. The Gershgorin disks are shown in blue, and the eigenvalues in red. The insets show the essential spectrum (magenta) as well as the eigenvalues. The largest connected component of Gershgorin disks consists of four or five overlapping disks, depending on the value of $\mu$, leading to two or four eigenvalues off the imaginary axis.}
    \label{fig:KdV Disks}
\end{figure}

Figure \ref{fig:KdV0} depicts the Gershgorin disks and the eigenvalues for $\mu=0$. Our analysis predicts that $k^*=2$. That means, $D_k(0)$ for $|k|>2$ are disjoint from all other disks, allowing at most five overlapping disks. The numerical results indeed confirm that there are five overlapping disks. 
Note that the zeroth Gershgorin disk has a radius of zero. While we have plotted it with a small but nonzero radius to make it visible, it remains challenging to discern. 

The overlapping Gershgorin disks enclose five eigenvalues, counted by algebraic multiplicity, specifically a triple eigenvalue at $0$ and two eigenvalues at approximately $\pm 2.44i$. 
The area covered by the three overlapping disks, which encompasses the complete `figure-8' spectral curve, is significantly larger than the area covered by the spectral curve itself. It is important to emphasize that the Gershgorin disks are guaranteed to contain the eigenvalues for $\mu=0$, but not necessarily the entire essential spectrum. 

Figures \ref{fig:KdV1}--\ref{fig:KdV5} show the Gershgorin disks for several values of $\mu$ between $0.1$ and $0.5$---corresponding to the antiperiodic boundary condition---along with the eigenvalues. 
For values of $\mu<0$, the spectrum can be obtained by symmetry after reflecting the results across the real axis. As $\mu$ increases, the triple eigenvalue at $0$ splits into three simple eigenvalues: two of them move into the complex upper half-plane, while the third moves down the imaginary axis. At $\mu=\mu_c \approx 0.24$, the eigenvalue moving down the imaginary axis collides with another eigenvalue coming up the imaginary axis, leading to a bifurcation of the pair into the complex plane. For the range $\mu \in (\mu_c,\frac12)$, there are four eigenvalues off the imaginary axis. When $\mu$ exceeds approximately $0.3$, the number of intersecting Gershgorin disks changes from five to four. 

Our analytical predictions are tight as far as the count is concerned. Indeed, $k^*=2$, ensuring that there can be at most $2k^*=4$ eigenvalues off the imaginary axis, which align well with the numerical findings. 

\subsection{The Kawahara equation}

We turn our attention to the Kawahara equation 
\begin{equation}\label{eqn:Kawahara0}
u_t=u_{xxxxx}+\alpha u_{xxx}+uu_x, \quad \alpha\in\mathbb{R},
\end{equation}
and the corresponding spectral problem 
\begin{equation}\label{eqn: Kawahara}
\lambda v=v_{xxxxx}+\alpha v_{xxx}+(Q(x)v)_x,\quad \lambda\in\mathbb{C},
\end{equation}
where $Q(x+T)=Q(x)$ for some $T>0$, the period. We set $c=0$ because focusing on stationary solutions suffices by Galilean invariance. Let 
\[
Q(x) = \sum_{k\in\mathbb{Z}} \widehat Q_k e^{\frac{2\pi k i x}{T}}\quad\text{and}\quad
\|\widehat Q\|_1 = \sum_{k\neq 0,\in\mathbb{Z}} |\widehat Q_k|,
\]
where we treat $\widehat{Q}_0$ separately. 

When $\alpha>0$, this scenario reflects competition between third-order and fifth-order dispersion terms, which is of particular interest. Since the dispersion relation is not monotone with the wave number, the results of the previous section may not directly apply without some adjustments, but the analysis remains elementary. 

For each $\mu \in (-\frac{1}{2}, \frac{1}{2}]$, the $k$-th Gershgorin disk for \eqref{eqn: Kawahara} and \eqref{eqn: Floquet quasiperiodic BCs} is defined as 
\begin{equation}\label{def:D(Kawaraha)}
\begin{aligned}
D_k(\mu)=\Big\{\lambda\in\mathbb{C} : \Big| \lambda-i\Big(\frac{2\pi}{T}(k+\mu)\Big)^5+&i\alpha\Big(\frac{2\pi}{T}(k+\mu)\Big)^3 \\
&- \frac{2\pi i }{T} (k+\mu)\widehat Q_0\Big| \leq \frac{2\pi}{T}|k+\mu|\|\widehat Q\|_1 \Big\}.
\end{aligned}
\end{equation}
See \eqref{eqn: General Gershgorin disc formula}, where the dispersion relation is $\omega(k)=k^5-\alpha k^3$. For each $\mu \in (-\frac{1}{2}, \frac{1}{2}]$, therefore, the eigenvalues of \eqref{eqn: Kawahara} and \eqref{eqn: Floquet quasiperiodic BCs} are contained in  $\bigcup_{k\in\mathbb{Z}} D_k(\mu)$. 

\begin{lemma}\label{Lemma:Kawaharadisjointcondition}
For each $\mu \in (-\frac{1}{2}, \frac{1}{2}]$, if
\begin{equation}\label{eqn:Kawahara bound}
\begin{aligned}
\Big|\Big(\frac{2\pi}{T}(k+1+\mu)\Big)^5
-&\alpha\Big(\frac{2\pi}{T}(k+1+\mu)\Big)^3 \\
-&\Big(\frac{2\pi}{T}(k+\mu)\Big)^5
+\alpha\Big(\frac{2\pi}{T}(k+\mu)\Big)^3 + \frac{2\pi}{T} \widehat Q_0\Big| \\
&\hspace*{60pt}>\frac{2\pi}{T}(|k+1+\mu|+|k+\mu|)\|\widehat Q\|_1
\end{aligned}
\end{equation}
for some $k\in\mathbb{Z}$, then $D_k(\mu)$ and $D_{k+1}(\mu)$ are disjoint, where $D_k(\mu)$ is in \eqref{def:D(Kawaraha)}. In particular, there exist $k^\ast$ and $k_\ast\in\mathbb{Z}$ such that $D_k(\mu)$ are disjoint provided $k>k^\ast$ or $k<k_\ast$.
\end{lemma}

See Theorem~\ref{Thm: General disjoint disc condition}.  
Lemma~\ref{Lemma:Kawaharadisjointcondition} provides a sufficient condition for the Gershgorin disks to be disjoint. The left side of \eqref{eqn:Kawahara bound} represents the distance between the centers of the $k$-th and $(k+1)$-st disks, which is a polynomial of degree four, while the right side corresponds to the sum of their radii and grows linearly. For sufficiently large $|k|$, the former is greater than the latter. 

We present some numerical experiments for \eqref{eqn:Kawahara0}, where $\alpha=2$, which admits explicit stationary solutions of the form
\[
u(x) = A_1 + A_2 \cn^2(\sigma x,m) + A_3 \cn^4(\sigma x,m),
\]
where $\cn(x,m)$ is the Jacobi elliptic function with the elliptic modulus $m$; $A_1$, $A_2$ and $A_3$ are constants, depending on $m$, and  $\sigma\in\mathbb{R}$ satisfies $|\frac{\alpha}{\sigma^2}|<52$. See, for instance, \cite{Bronski.Hur.Marangell}, for further details. We set $m=.6185$, resulting in $A_1=0.659$, $A_2=2.306$, $A_3=-2.51$. 

\begin{figure}
\centering
    \begin{subfigure}[b]{.4\columnwidth}
\includegraphics[width=\textwidth]{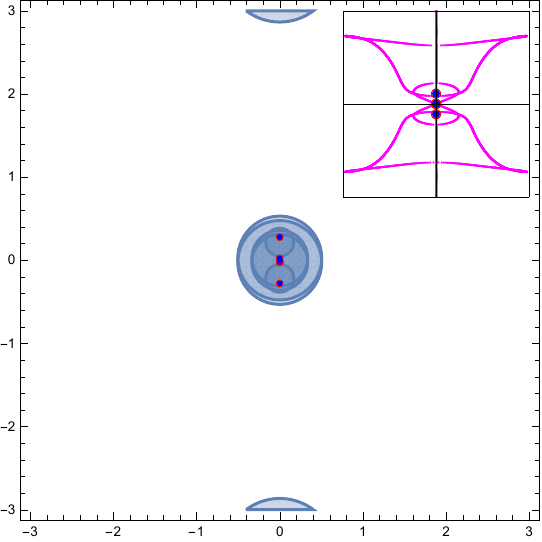}
        \caption{$\mu=0$}
        \label{fig:Kawa0}
        \end{subfigure}
        \begin{subfigure}[b]{.4\columnwidth}
\includegraphics[width=\textwidth]{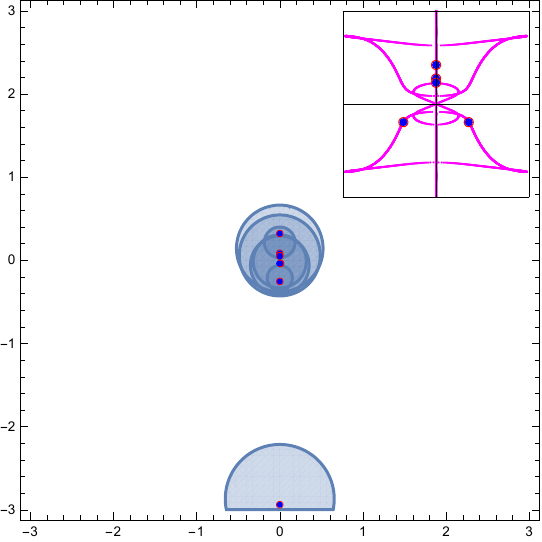}
        \caption{$\mu=0.1$}
        \label{fig:Kawa0}
        \end{subfigure}
        \begin{subfigure}[b]{.4\columnwidth}
\includegraphics[width=\textwidth]{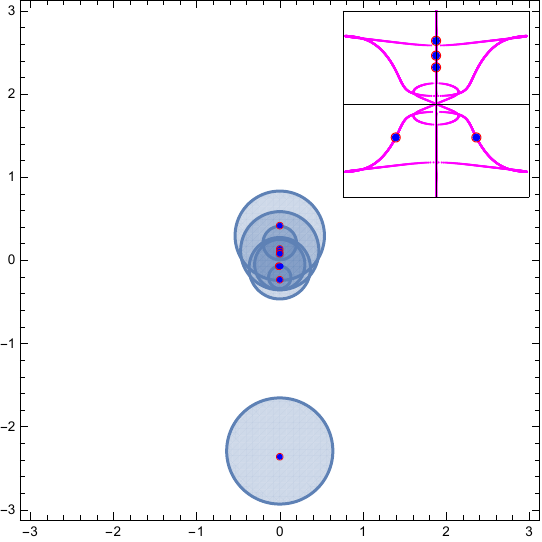}
        \caption{$\mu=0.2$}
        \label{fig:Kawa0}
        \end{subfigure}
        \begin{subfigure}[b]{.4\columnwidth}
\includegraphics[width=\textwidth]{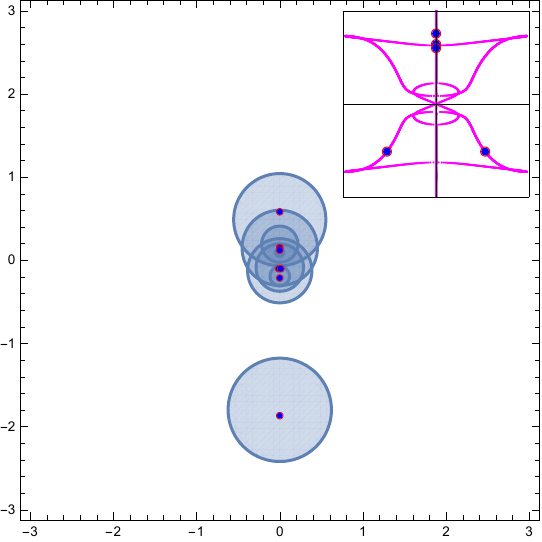}
        \caption{$\mu=0.3$}
        \label{fig:Kawa0}
        \end{subfigure}
        \begin{subfigure}[b]{.4\columnwidth}
\includegraphics[width=\textwidth]{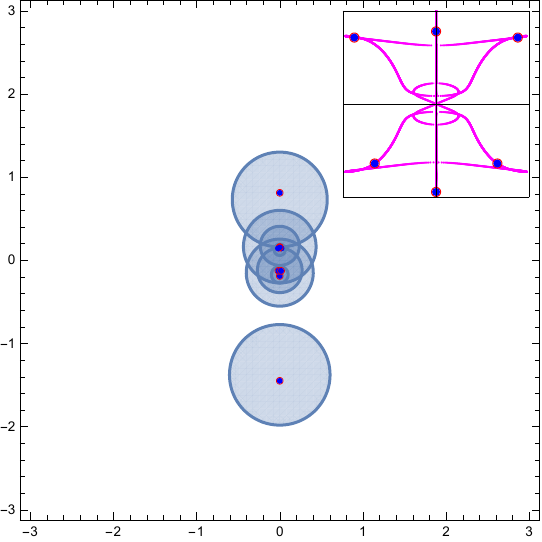}
        \caption{$\mu=0.4$}
        \label{fig:Kawa0}
        \end{subfigure}
        \begin{subfigure}[b]{.4\columnwidth}
\includegraphics[width=\textwidth]{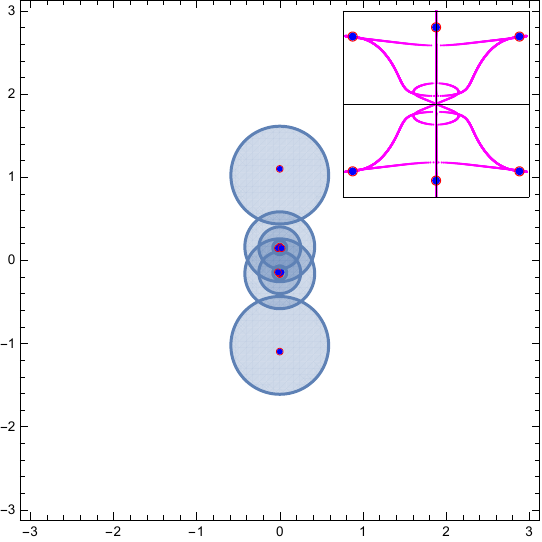}
        \caption{$\mu=0.5$}
        \label{fig:Kawa0}        
    \end{subfigure}
    \caption{The Gershgorin disks (blue), the eigenvalues (red), and the essential spectrum for the linearized Kawahara equation about a stationary periodic solution, for $\mu=0$, $0.1$, $0.2$, $0.3$, $0.4$, and $0.5$. The insets highlight the essential spectrum (magenta) and eigenvalues in the vicinity of $0$, where $\operatorname{Re}(\lambda)\in(-0.025,0.025)$ and $\operatorname{Im}(\lambda) \in (-0.2,0.2)$.}
    \label{fig:KawaharaDisks}
\end{figure}

Figure \ref{fig:KawaharaDisks} depicts the Gershgorin disks and the eigenvalues for $\mu=0$, $0.1$, $0.2$, $0.3$, $0.4$ and $0.5$, along with the essential spectrum. Over most of the range $\mu\in(-\frac12,\frac12]$, there exists a connected component consisting of seven intersecting disks, except for a small interval of near $\pm\frac{1}{2}$ where there are eight intersecting disks. 

At $\mu=0$, there are seven eigenvalues in the connected component, counted by algebraic multiplicity, specifically $\lambda \approx \pm .28i$, $\pm .022i$, and $0$ with multiplicity three. As $\mu$ increases, the triple eigenvalue at $0$ splits into two complex eigenvalues in the upper half-plane and one eigenvalue on the negative imaginary axis. As $\mu$ increases further, various collisions occur, causing the number of eigenvalues off the imaginary axis to vary between two and four. Since there are either seven or eight intersecting Gershgorin disks over the range $\mu\in(-\frac12,\frac12]$, the number of potential off-axis eigenvalues could a priori be as high as eight, although it never appears to exceed four.

\begin{remark*}[The Benjamin-Ono equation]\rm
An intriguing example is the Benjamin-Ono (BO) equation
\[
u_t+\mathcal{H}u_{xx}+uu_x=0, 
\]
where $\mathcal{H}$ denotes the Hilbert transform, defined via the Fourier transform as $\widehat{\mathcal{H}v}(k)=-i\operatorname{sgn}|k|\widehat{v}(k)$. The corresponding spectral problem takes the form
\begin{equation}\label{eqn:BO}
\lambda v+\mathcal{H}v_{xx}-cv_x+(Q(x)v)_x=0, \quad\lambda\in\mathbb{C},
\end{equation}
where $c\in\mathbb{R}$ is a constant and $Q(x)$ is a real-valued function satisfying $Q(x+T)=Q(x)$ for some $T>0$, the period. Here we do not assume that $Q(x)$ has zero mean. 

Considering \eqref{eqn:BO} and \eqref{eqn: Floquet quasiperiodic BCs}, \eqref{eqn: General Gershgorin disc formula} becomes
\begin{equation}\label{eqn: BO Dk(mu)}
D_k(\mu)=
\Big\{\lambda\in\mathbb{C} :\Big|\lambda-i\Big(\frac{2\pi}{T}\Big)^2\operatorname{sgn}(k+\mu)(k+\mu)^2-\frac{2\pi ic}{T}(k+\mu)\Big| 
\leq \frac{2\pi}{T}|k+\mu|\|\widehat{Q}\|_1 \Big\},
\end{equation}
where $k\in\mathbb{Z}$ and $\mu \in(-\frac12,\frac12]$. 

Additionally, we compute 
\begin{equation}\label{eqn:Ben-Ono Q Norm}
\|\widehat{Q}\|_1=\frac{\frac{8\pi^2}{cT^2}}{1-\sqrt{1-\frac{4 \pi^2}{c^2 T^2}}}.
\end{equation}
Recall from \cite{ono1975algebraic} that periodic traveling waves of the BO equation are explicitly given by 
\[
Q(x)=\frac{A}{1- B \operatorname{cos} \left(\frac{2 \pi x}{T} \right)},
\qquad A=\frac{8 \pi^2}{cT^2}\quad\text{and}\quad B= \sqrt{1- \frac{4 \pi^2}{T^2 c^2}}.
\]
We assume $2\pi<cT$, so that $B$ is real-valued, and assume $c>0$. The Fourier coefficients are calculated via the contour integration, yielding 
\[
\widehat{Q}_k=\frac{A}{2 \pi}\int \limits_0^{2 \pi} \frac{e^{ i k \xi } }{1- \frac{B}{2} (e^{i \xi} +e^{-i \xi})}~d \xi
=\frac{A}{\sqrt{1-B^2}} \Big(\frac{1-\sqrt{1-B^2}}{B}\Big)^{|k|}.
\]
Recalling $\|\widehat{Q}\|_1=\sum_{-\infty}^\infty |\widehat{Q}_k|$, \eqref{eqn:Ben-Ono Q Norm} follows after evaluating the geometric series. 

We claim that the Gershgorin disks for the BO equation (see \eqref{eqn: BO Dk(mu)}) are never disjoint except potentially at one point. Indeed, $D_k(\mu)$ is centered at $i\big(\frac{2\pi}{T}\big)^2\operatorname{sgn}(k+\mu)(k+\mu)^2+O(k)$ as $|k|\to\infty$, implying that the distance between adjacent disk centers grows asymptotically like $\frac{8 \pi^2}{T^2}|k|$ as $|k|\to\infty$, while the radius of $D_k(\mu)$ is $\frac{2\pi}{T}|k+\mu|\|\widehat Q\|_1$. For these disks to be disjoint, 
\[
\frac{8 \pi^2k}{T^2}>\frac{2 \pi k}{T}\|\widehat{Q}\|_1 
\]
must hold, which simplifies to 
\[
\frac{T \|\widehat{Q}\|_1}{4 \pi}<1.
\]
Substituting \eqref{eqn:Ben-Ono Q Norm}, on the other hand,
\[
\frac{x}{1-\sqrt{1-x^2}}<1, \qquad x=\frac{2 \pi}{c T}
\]
must hold, which is impossible except potentially at one point where equality might hold. At such a point, an appropriate choice of parameters could allow the terms of order $k$ in the disk centers to align in a way that the disks are eventually disjoint.  

Although the Gershgorin disk theorem argument does not offer the desired bound for the BO equation, numerical experiments (see \cite{SarahThesis} for details) suggest that the spectrum for the BO equation is confined to regions significantly smaller than theoretical predictions. 
\end{remark*}

\begin{remark*}\rm
Our results apply to equations of the form \eqref{eqn: General eigenvalue problem}, where the dispersion relation satisfies $\omega(k)=O(|k|^d)$ as $|k| \to \infty$ for some $d>2$, so that the $k$-th Gershgorin disk is centered at a $O(k^d)$ distance from $0\in\mathbb{C}$ along the imaginary axis, with a $O(k^{d-1})$ distance to adjacent disks, and a radius of $O(k)$ as $|k| \to \infty$. For these disks to remain asymptotically disjoint as $|k| \to \infty$, it is necessary that the distance between adjacent disks grows at a rate greater than that of the radius, requiring $d>2$. 

It is important to note that our results are not directly applicable to nonlinear dispersive equations of the form
\[
u_t+\omega(-i\partial_x)u_x+uu_x=0,
\]
where $\omega(k)=O(|k|^d)$ as $|k| \to \infty$ for some $d<1$. Notable examples include the Whitham equation \cite{whitham2011linear}, for which $\omega(k)=\sqrt{\frac{\tanh(k)}{k}}$, and the gravity-capillary Whitham equation \cite{HJcapWhitham} with $\omega(k)=O(|k|^{1/2})$ for $|k|\gg 1$. Additionally, our approach does not directly extend to the nonlinear Schr\"odinger equation. See \cite{SarahThesis} for further details. 
\end{remark*}

\subsection{The BBM equation}

Lastly, we consider the BBM equation
\[
u_t=u_{xxt}+uu_x,
\]
and the corresponding spectral problem 
\begin{align}\label{eqn:BBM}
\lambda v = c v_{x} + (1-\partial_{x}^2)^{-1} (Q(x) v)_x,
\end{align}
where $c\in\mathbb{R}$, and $Q(x+T)=Q(x)$ some $T>0$, the period. 
In previous sections, we have assumed that $Q(x)$ has zero mean, incorporating the mean term into $c$. Due to the structure of the symplectic operator involving $(1-\partial_{x}^2)^{-1}$, however, this is no longer convenient, and here we do not assume that $Q(x)$ has zero mean.

For each $\mu \in (-\frac{1}{2}, \frac{1}{2}]$, \eqref{eqn: General Gershgorin disc formula} becomes
\begin{equation}\label{eqn:D(BBM)}
D_k(\mu)=\Big\{\lambda\in\mathbb{C}:\Big|\lambda- \frac{2\pi ic}{T}(k+\mu)\Big| 
< \frac{\frac{2\pi }{T}|k+\mu|}{1+\big(\frac{2\pi}{T}\big)^2(k+\mu)^2}\|\widehat{Q}\|_1\Big\},\qquad k\in\mathbb{Z},
\end{equation}
and any eigenvalues of \eqref{eqn:BBM} and \eqref{eqn: Floquet quasiperiodic BCs} must belong to $\bigcup_{k\in\mathbb{Z}} D_k(\mu)$.

The proof of \eqref{eqn:D(BBM)} follows a similar methodology used for the gKdV and Kawahara equations. We omit the details. It is noteworthy that the scaling is slightly different here. For equations of KdV type, the spacing between Gershgorin disks grows algebraically for large $|k|$, while the radii of the disks increase linearly. In contrast, here the disks are equally spaced, and the radii tend to zero as $|k|\to\infty$.  

\begin{lemma}\label{lemma: BBMsuff}
For sufficiently small $\|\widehat Q\|_1$, all Gershgorin disks in \eqref{eqn:D(BBM)} are disjoint for all $\mu\in(-\frac12,\frac12]$. Specifically, if
\begin{equation}\label{eqn:BBMAbsolute}
\frac{2 \pi |c|}{T} > \| \widehat Q\|_1
\end{equation}
then $D_k(\mu)\cap D_{k'}(\mu)=\emptyset$ for all $k\neq k'$ and all $\mu\in(-\frac12,\frac12]$.
Moreover if \eqref{eqn:BBMAbsolute} does not hold then
\begin{equation}\label{eqn: BBM disjoint condition 2}
k+\mu> \frac{T^2}{2\pi^2|c|}\|\widehat{Q}\|_1
\end{equation} 
implies that $D_k(\mu)$ and $D_{k+1}(\mu)$ are disjoint. In particular, there exist $k^\ast$ and $k_\ast \in \mathbb{Z}$ such that $D_k(\mu)$ are disjoint for $k>k^\ast$ or $k<k_\ast$. 
\end{lemma}

We emphasize that these conditions are designed to be simple to state and are not necessarily the sharpest possible.

\begin{proof}
Note that 
\[
\frac{|x|}{1+x^2} \leq \frac12
\]
implies that all Gershgorin disks have a radius $<\frac12\|\widehat Q\|_1$. Since the distance between the centers of two adjacent disks is $\frac{2\pi |c|}{T}$, \eqref{eqn:BBMAbsolute} follows. 

Additionally, note that the function $\frac{x}{1+x^2}$ is decreasing for $x>1$, so that the radius of the $(k+1)$-st disk is smaller than the radius of the $k$-th disk if $\frac{2\pi}{T}(k+\mu)>1$. Consequently, for the Gershgorin disks to be disjoint, 
\[
\frac{2\pi c}{T} > 2 \frac{x}{1+x^2} \|\widehat Q\|_1
\]
which simplifies to 
\[
1 > \frac{b x}{1+x^2},\quad 
\text{$b= \frac{\|\widehat Q\|_1 T}{\pi c}$ and $x=\frac{2\pi}{T}(k+\mu)$}.
\]
Clearly, $1 > \frac{b x}{1+x^2}$ when $x>b$, and if \eqref{eqn:BBMAbsolute} does not hold, then necessarily $b>2$, implying monotonicity. This completes the proof.
\end{proof}

\begin{figure*}
    \centering 
    \begin{subfigure}[b]{.4\textwidth}
        \includegraphics[width=\textwidth]{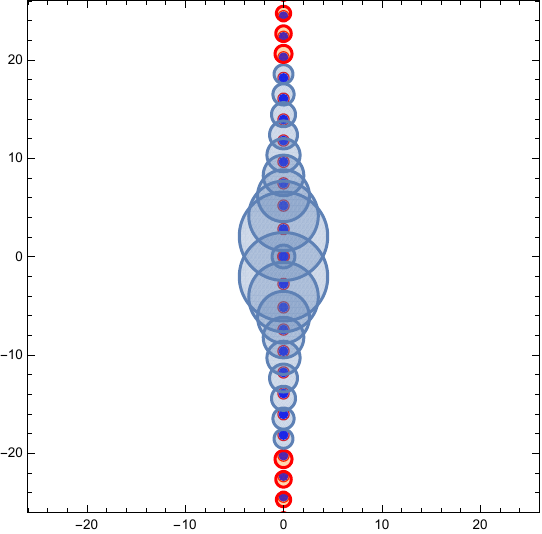}
        \caption{$\mu=0$}
        \label{fig:BBMMu0}
    \end{subfigure}    
    \begin{subfigure}[b]{.4\columnwidth}
        \includegraphics[width=\textwidth]{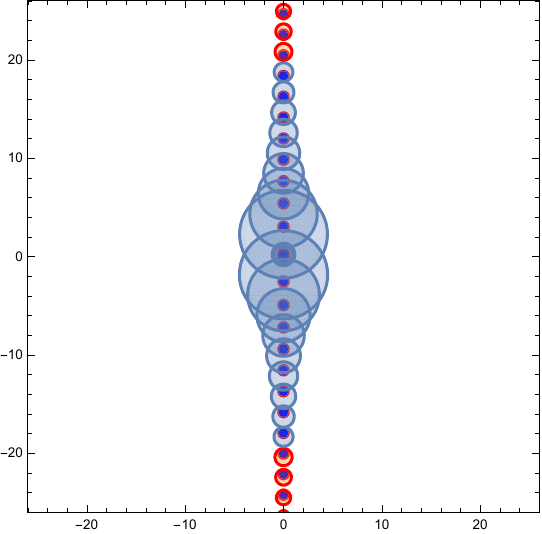}
        \caption{$\mu=0.1$}       \label{fig:BBMMu1}
    \end{subfigure}
        \begin{subfigure}[b]{.4\columnwidth}
        \includegraphics[width=\textwidth]{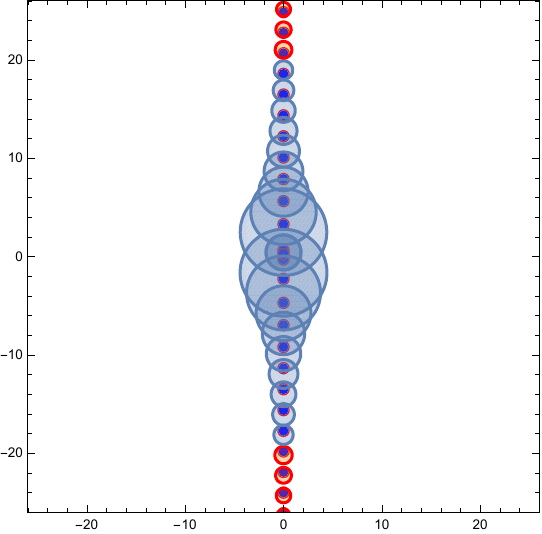}
        \caption{$\mu=0.2$}
        \label{fig:BBMMu2}
    \end{subfigure}
        \begin{subfigure}[b]{.4\textwidth}
        \includegraphics[width=\textwidth]{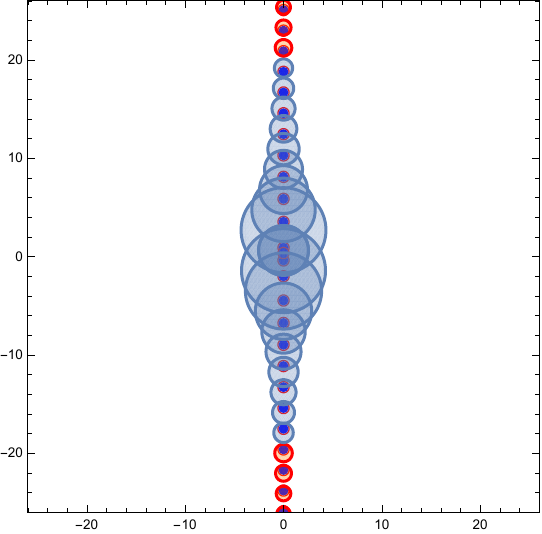}
        \caption{$\mu=0.3$}
        \label{fig:BBMMu3}
    \end{subfigure}    
    \begin{subfigure}[b]{.4\columnwidth}
        \includegraphics[width=\textwidth]{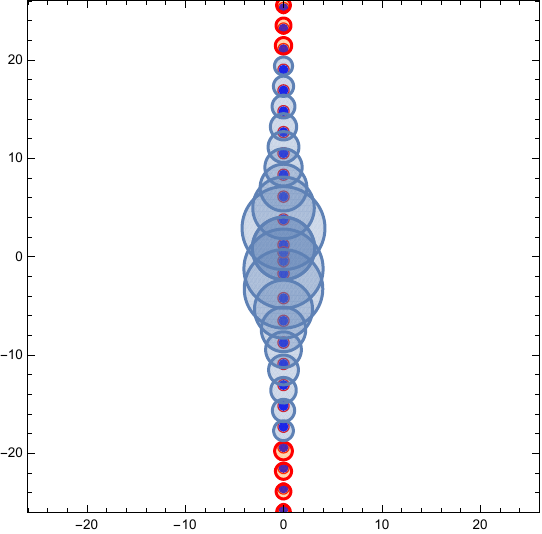}
        \caption{$\mu=0.4$}
        \label{fig:BBMMu4}
    \end{subfigure}
        \begin{subfigure}[b]{.4\columnwidth}
        \includegraphics[width=\textwidth]{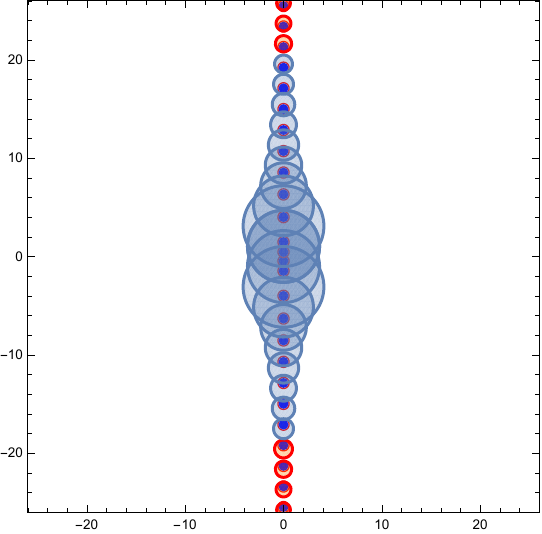}
        \caption{$\mu=0.5$}
        \label{fig:BBMMu5}
    \end{subfigure}
   \caption{The Gershgorin disks and the spectrum for \eqref{eqn:BBM}, for $\mu=0,0.1,0.2,0.3,0.4,0.5$. The disks outlined in red ($|k|>9$) are disjoint from all other disks for all values of $\mu$.}
    \label{fig:BBMDiscs}
\end{figure*}

The BBM equation admits a family of periodic traveling waves of the form 
\[
u(x,t) = -\frac{6 m}{2m-1}\cn^2\Big(\frac{x-2t}{2\sqrt{2m-1}},m\Big) 
\]
for the elliptic parameter $m$ in the range $(\frac12,1]$. By rescaling, all non-constant traveling waves can be normalized to have a wave speed of $\pm2$. We consider the case $m=\frac34$, which gives a periodic traveling wave oscillating between $0$ and $-9$ with a period $2\sqrt{2}\K(\frac34)\approx 6.0996$, where $\K(m)$ denotes the complete elliptic integral of the first kind. Since all the Fourier coefficients of the potential have the same sign, we find that $\|\widehat Q\|_1=\|Q\|_\infty=9$. Applying the result of Lemma \ref{lemma: BBMsuff}, we numerically find that the Gershgorin disks are disjoint provided 
\[
k+\mu>\frac{T^2 \|\widehat Q\|_1}{2\pi^2 |c|} = \frac{9(6.0996)^2}{4\pi^2}\approx 8.48.
\]
Since $\mu\in (-\frac12,\frac12]$, this implies that the disks are disjoint for all $|k|>9$ for all $\mu$.

Figure \ref{fig:BBMDiscs} presents some numerical results, specifically, showing $D_k(\mu)$ for $k=-12,-11$, $\dots$, $11,12$ for several values of $\mu \in [0,\frac12]$. The disks marked in red (corresponding to $|k|>10$) are disjoint from all other disks for all values of $\mu$, while the remaining disks intersect another disk for at least some range of $\mu$. The estimate of $k^*$ is therefore quite tight, although in this case there are no instabilities and the spectrum remains confined to the imaginary axis. 

\section{Concluding remarks}\label{sec:conclusion}

We have introduced a novel methodology applicable to a broad class of dispersive Hamiltonian PDEs, for establishing bounds on regions where the spectrum of the corresponding quasi-periodic eigenvalue problem may potentially deviate from the imaginary axis. We also provide estimates for the number of such unstable eigenvalues. When the dispersion relation grows strictly faster than quadratically in the wave number, we adopt a Gershgorin disk theorem argument, along with the symmetry properties of the spectrum. 

We would be remiss if we did not elaborate on the similarities and differences between the present study and the recent work in \cite{gaebler2021nls}. One can summarize the methods in \cite{gaebler2021nls} are primarily semigroup theoretic. Specifically, for the gKdV equation (see \eqref{eqn:gKdV}), the authors in \cite{gaebler2021nls} bound the symmetric part of the operator, leading to the conclusion (see \cite[Theorem 3]{gaebler2021nls}) that the eigenvalues off the imaginary axis lie within a strip
\[
\|\operatorname{Im}(\lambda)\| \lesssim \| Q\|_\infty^3.
\]
This is notably similar to our Corollary \ref{cor: General disjoint disc condition}, which establishes that any eigenvalues off the imaginary axis must have a wave number $\lesssim \|\widehat Q \|_1$. Since the centers of the Gershgorin disks for the gKdV equation are located at $k^3 + O(k^2)$ as $|k|\to\infty$, this implies that the off-axis eigenvalues must lie in a strip of width $\|\widehat Q \|_1^3$. 

It is generally true that $\|Q\|_\infty \leq \| \widehat Q \|_1$. While there do exist $L^\infty$ functions for which the Fourier coefficients are not absolutely summable, for smooth potentials arising in the stability analysis for periodic traveling waves, one expects that the two norms are of the same order. In cases involving elliptic functions, which are prevalent in many applications, the Fourier coefficients either have a fixed sign or alternate in sign, so that the $\ell_1$ norm of the Fourier coefficients matches exactly the $L^\infty$ norm of the potential. 

On the other hand, the two approaches are mutually complementary. The main objective of \cite{gaebler2021nls} is to derive bounds on the operator norm of the semigroup generated by the linearized Hamiltonian flow. This is accomplished in \cite[Theorem 3]{gaebler2021nls}, assuming knowledge on $N$, the dimension of the largest Jordan block of the linearized operator. It is not evident, however, if such quantities can be estimated using purely semigroup theoretic techniques. While it follows from general arguments that the number must be finite and, indeed, the eigenvalues lie in a bounded region and cannot have finite accumulation points, no a priori control is available. 
On the other hand, such an estimate {\em does follow from} our Corollary \ref{cor: General disjoint disc condition}. Since all eigenvalues with $|k|>k^*$ are necessarily simple, it follows that $N\leq 2k^*+1$. This enables a semigroup estimate entirely in terms of a norm of the potential without the need to explicitly solve the spectral problem.

\bibliographystyle{amsplain}
\bibliography{refs}

\end{document}